\documentclass{article}

\usepackage{amsmath, amsfonts, comment, bbm}
\usepackage{amssymb, hyperref}
\usepackage{tikz-cd}
\usepackage{enumitem}

\usepackage{amsthm}
\newtheorem{theorem}{Theorem}[section]
\newtheorem{corollary}[theorem]{Corollary}
\newtheorem{lemma}[theorem]{Lemma}
\newtheorem{proposition}[theorem]{Proposition}
\newtheorem*{thm:dowkerequivalence}{Theorem \ref{thm:dowkerequivalence}}
\newtheorem*{thm:dowkerduality}{Theorem \ref{thm:dowkerduality}}
\theoremstyle{definition}
\newtheorem{definition}[theorem]{Definition}
\newtheorem*{thm:dowkerrelation}{Definition \ref{thm:dowkerrelation}}
\newtheorem*{thm:dowkerfiber}{Definition \ref{thm:dowkerfiber}}
\newtheorem{remark}[theorem]{Remark}

\newcommand{\set}{\texttt{Set}}
\newcommand{\rel}{\texttt{Rel}}
\newcommand{\Top}{\texttt{Top}}
\newcommand{\cat}{\texttt{Cat}}

\newcommand{\sset}{\mathcal{S}}
\newcommand{\ssset}{\sset^2}

\newcommand{\op}{^\texttt{op}}
\newcommand{\ob}[1]{\text{ob}(#1)}

\newcommand{\id}{\mathbbm{1}}

\newcommand{\sing}{\text{Sing}}

\newcommand{\transformation}{transformation}

\newcommand{\R}{\mathcal{R}}
\newcommand{\RR}{\mathbb{R}}
\newcommand{\C}{\mathcal{C}}
\newcommand{\A}{\mathcal{A}}
\newcommand{\D}{\mathcal{D}}

\newcommand{\Vtr}{\mathcal{V}}
\newcommand{\Ns}{N}

\newcommand{\tw}{\operatorname{tw}}
\newcommand{\oD}{\operatorname{diag}}
\newcommand{\p}{\operatorname{p}\!}
\newcommand{\oP}{\operatorname{P}\!}
\newcommand{\od}{d}

\newcommand{\V}{\mathcal{V}}

\usepackage[british]{babel}
\usepackage[useregional]{datetime2}
\DTMlangsetup[en-GB]{showdayofmonth=false}

\usepackage{authblk}

\title{Dowker Duality for Relations of Categories}
\author[1]{Morten Brun}
\author[1]{Marius Gårdsmann Fosse}
\author[1]{Lars M. Salbu}
\affil[1]{Department of Mathematics, University of Bergen.}

\begin{document}

\maketitle
\begin{abstract}
    We propose a categorification of the Dowker duality theorem for
    relations. Dowker's theorem states that the Dowker complex of a relation \(R \subseteq X \times Y\)
    of sets \(X\) and \(Y\) is homotopy equivalent to the Dowker complex of the transpose relation
    \(R^T \subseteq Y \times X\).
    Given a relation \(R\) of small categories \(\C\) and \(\D\), that is, a functor of the form \(R \colon \R \to \C \times \D\), we 
    define the {\em bisimplicial rectangle nerve \(ER\)} and the 
    \emph{Dowker nerve} \(DR\). The diagonal \(d(ER)\) of the bisimplicial set
    \(ER\) maps to the simplicial set \(DR\) by a 
    natural projection \(d(\pi_R) \colon d(ER) \to DR\).
    
    We introduce a criterion on relations of categories ensuring that the projection from the diagonal of the bisimplicial rectangle nerve to the Dowker nerve is a weak equivalence. Relations satisfying this criterion are called \emph{Dowker relations}.
    If both the relation \(R\) of categories and its transpose relation
    \(R^T\) are Dowker relations, then the Dowker nerves \(DR\) and \(DR^T\) 
    are weakly equivalent simplicial sets.

    In order to justify the abstraction introduced by our categorification
    we give two applications.
    The first application is to show that Quillen's Theorem A can be
    considered as an instance of Dowker duality.
    In the second application we consider a simplicial complex $K$ with vertex set $V$ and show that the geometric realization of $K$ is naturally homotopy equivalent to the geometric realization of the simplicial set with the set of $n$-simplices given by functions $\{0,1,\dots,n\}\to V$ whose image is a simplex of $K$.
\end{abstract}

\section{Introduction}

In the paper
``Homology Groups of Relations''\cite{Dowker} from 1952,
C.H. Dowker associates an abstract simplicial complex \(D(R)\) to a relation $R\subseteq
X\times Y$ from a set \(X\) to a set \(Y\). The vertex set of \(D(R)\) is the
set \(X\), and a subset \(\sigma\) of \(X\) is a simplex in \(D(R)\) if and only
if there exists an element \(y \in Y\) such that \(\sigma \times \{y\} \subseteq
R\).
Dowker's theorem \cite[Theorem 1a, p.
89]{Dowker} states that the homology groups of the {\em Dowker complex} $D(R)$ are
isomorphic to the homology groups of the Dowker complex $D(R^T)$ of the
transposed relation $R^T \subseteq Y \times X$ consisting of pairs \((y,x)\) with \((x,y) \in R\).

Before introducing our categorification of the Dowker duality theorem, we give
a short summary of its history.
In \cite[Theorem 10.9]{bjorner} Bj\"orner shows that the simplicial complexes
$D(R)$ and $D(R^T)$ have homotopy equivalent geometric realizations by
constructing an explicit homotopy equivalence $\varphi_R \colon \vert D(R)\vert
\to \vert D(R^T)\vert$.  Given an inclusion $R \subseteq S$ of relations from
\(X\) to \(Y\),  Chowdhury and Mémoli \cite[Theorem 3]{memoli} shows that the
diagram \begin{equation*} \begin{tikzcd} \vert D(R)\vert
\arrow[d]\arrow[r,"{\varphi_R}"]&\vert D(R^T)\vert\arrow[d]\\ \vert D(S)\vert
\arrow[r,"{\varphi_S}"]&\vert D(S^T)\vert  \end{tikzcd} \end{equation*}
commutes up to homotopy, giving a \emph{functorial} Dowker theorem. In
\cite[Theorem 5.2]{Virk} Virk extends this result to morphisms \(f \colon R \to
R'\) of relations \(R \subseteq X \times Y\) and \(R' \subseteq X' \times Y'\)
given by a pair \((f_1,f_2)\) of functions \(f_1 \colon X \to X'\) and \(f_2
\colon Y \to Y'\) such that the image of \(R\) under the function \(f_1 \times
f_2 \colon X \times Y \to X' \times Y'\) is contained in \(R'\). 
Brun and Salbu give an alternative proof of the functorial Dowker theorem in \cite{brunsalbu} by introducing the {\em rectangle complex}
\(E(R)\) of the relation \(R \subseteq X \times Y\).
The assignment \(R \mapsto E(R)\) is
a functor from the category of relations with morphisms of the above form to
the category of simplicial complexes.  The projection \(X \times Y \to X\)
induces a natural map \(E(R) \to D(R)\) and the projection \(X \times Y \to Y\)
induces a natural map \(E(R) \to D(R^T)\). The functorial Dowker theorem is
proven by showing that the geometric realizations of these maps are homotopy
equivalences.

In this paper we consider relations \(R\) from a small category \(\C\)
to a small category \(\D\), that is, functors of the form \(R \colon \R \to \C \times
\D\). Such functors are usually
called a \emph{span}, but guided by the work of Dowker we call them
\emph{relations of categories}. The aim of this paper is to propose a version of Dowker's
Theorem for relations of this form.

Given a relation \(R \colon \R \to \C \times \D\), we introduce the
{\em bisimplicial rectangle nerve \(ER\)}. It is a bisimplicial set whose set
$ER_{m,n}$ of \((m,n)\)-simplices consists of functors of the form \(r \colon
[m] \times [n] \to \R\) with the property that there exist functors \(a \colon
[m] \to \C\) and \(b \colon [n] \to \D\) such that \(R \circ r = a \times b\). Here $[m]$ is the totally ordered set $\{0<1<\dots<m\}$ considered as a category.
If such functors \(a\) and \(b\) exist, they are uniquely determined. This
implies that there is a map \(\pi_R \colon ER_{m, n} \to N\C_m\), into the
\(m\)-simplices \(N\C_m\) of the nerve of \(\C\), taking \(r\) as above to
\(\pi_R(r) = a\).  In this context, the {\em Dowker nerve \(DR\)} is the
simplicial subset of \(N\C\) with \(m\)-simplices given by the image of the map
\(\pi_R \colon ER_{m,0} \to N\C_m\).
The bisimplicial rectangle nerve 
is our categorification of the rectangle complex of \cite{brunsalbu}.
In Section
\ref{bpdd} we prove our main results. In order to state them we need two definitions from that section.
\begin{thm:dowkerfiber}
    Given \(a \in N\C_m\), that is, a functor \(a \colon [m] \to \C\), the {\em
    fiber \(\pi_R^a\) of \(a\) under \(\pi_R\)} is the simplicial subset of
    the simplicial set \([n] \mapsto ER_{m, n}\) consisting of functors \(r
    \colon [m] \times [n] \to \R\) such that there exist a functor \(b \colon
    [n] \to \D\) with \(R \circ r = a \times b\).
\end{thm:dowkerfiber}
\begin{thm:dowkerrelation}
    A {\em Dowker relation} is a relation \(R \colon \R \to \C \times \D\)
    with the property that for every \(a \in N\C_m\), the fiber \(\pi_R^a\) of
    \(a\) under \(\pi_R\) is contractible or empty.
\end{thm:dowkerrelation}
The {\em transpose relation \(R^T\)} of a relation \(R \colon \R \to \C \times
\D\) is the composite \(\tw \circ R \colon \R \to \D \times \C\)
of \(R\) and the twist isomorphism \(\tw \colon \C \times \D \to \D \times
\C\).
There is an isomorphism \(\tw^* \colon ER_{m,n} \to ER^T_{n,m}\)
taking \(r \colon [m] \times [n] \to \R\) to the composition \(r \circ \tw\) of
\(\tw \colon [n] \times [m] \to [m] \times [n]\) and \(r\).
The diagonal simplicial set \(d(ER)\) of \(ER\) is the simplicial set with
\(n\)-simplices given by the set \(d(ER)_n = ER_{n,n}\).
\begin{thm:dowkerduality}[Dowker Equivalence]
    If \(R\) is a Dowker relation, then the projection maps
    \[\pi_{R} \colon ER_{m,n} \to DR_m\]
    induce a weak equivalence \(\od(\pi_{R}) \colon
    d(ER) \to DR\) of simplicial sets.
\end{thm:dowkerduality}
In Section \ref{bsrc} we introduce morphisms of relations.
The following is our version of Dowker's duality theorem:
\begin{thm:dowkerequivalence}[Dowker Duality]
Given a morphism $f:R\to R'$ of relations of categories, there is a commutative diagram of the form
\begin{equation*}
    \begin{tikzcd}
    DR\arrow[d,"Df",swap]&\arrow[l,"\od(\pi_{R})",swap]\od(ER)\arrow[d,"\od(Ef)"]\arrow[r,"\od(\tw^*)"]&\od(ER^T)\arrow[d,"\od(Ef^T)"]\arrow[r,"\od(\pi_{R^T})"]&DR^T\arrow[d,"Df^T"]\\
    DR'&\arrow[l,"\od(\pi_{R'})",swap]\od(ER')\arrow[r,"\od(\tw^*)"]&\od(ER'^T)\arrow[r,"\od(\pi_{R'^T})"]&DR'^T.
    \end{tikzcd}
\end{equation*}
    If the relations \(R, R^T, R'\) and \({R'}^T\) are Dowker relations, then all
    horizontal maps in this diagram are weak equivalences of simplicial sets.
\end{thm:dowkerequivalence}

We end the paper with two applications of Theorems \ref{thm:dowkerduality} and \ref{thm:dowkerequivalence}.
Given functors of the form \(F \colon \C \to \A\) and \(G \colon \D \to \A\),
the projection \(R \colon F
\downarrow G \to \C \times \D\)  taking an object \((c,d,f)\) of the comma
category \(F \downarrow G\) to \((c,d)\) is a relation. We show that if the
nerve of the
category \(F
\downarrow d\) is contractible for every object \(d\) of \(\D\),
then \(R\) is a Dowker relation. Using this we basically recover Quillen's original proof of his Theorem A \cite[Theorem A]{quillen}.

As a second application we show that 
the geometric realization of a simplicial complex
\(K\) is naturally homotopy equivalent to the geometric realization of 
the \emph{singular complex} \(\sing(K)\), a simplicial set defined as follows:
Let \(V\) be the vertex set of \(K\).
The set of \(m\)-simplices of \(\sing(K)\) is the set of functions
\(\{0,1,\dots,m\} \to V\) whose image is a simplex in \(K\). 

The assignment \(K \mapsto \sing(K)\) is a functor from the category of simplicial complexes to the category of
simplicial sets. The geometric realization of \(\sing(K)\) is much 
bigger than the geometric realization of \(K\). There are other smaller simplicial sets that capture the homotopy type of the geometric realization of \(K\). One example is the nerve $N(K_\subseteq)$ of the category \(K_{\subseteq}\) given by \(K\) considered as a partially ordered set under inclusion. The assignment $K\mapsto N(K_\subseteq)$ is also a functor, with the convenient property that the geometric realizations of $K$ and $N(K_\subseteq)$ are naturally homeomorphic. This functor has neither a left- nor a right adjoint functor. In contrast, the singular complex \(K \mapsto \sing(K)\) has a left adjoint functor.

The fact that the geometric realizations of $K$ and  $\sing(K)$ are
homotopy equivalent is a well-known fact in topology, but to the best of our knowledge it has not yet been published in a peer-reviewed paper. Two proofs of this fact have been published on the personal web page of Omar Antolín Camarena \cite{omar}, but the naturality of the homotopy equivalence is lacking as both proofs use a chosen order on the vertex set of $K$. 


The paper is organized as follows:
In Section \ref{bsrc} we give preliminary definitions concerning (bi)simplicial sets and we introduce relations of categories. In Section \ref{sec:} we define the bisimplicial rectangle nerve, which in Section \ref{bpdd} we use to prove our Dowker Equivalence and Dowker Duality theorems. In Section \ref{sec:homotopies} we look at homotopies of Dowker nerves, and finally, in Section \ref{sec:appl} we present the two applications of our main results presented in the two preceding paragraphs.

\section{Bisimplicial Sets and Relations of Categories}\label{bsrc}
In this section we recall the definition of simplicial and bisimplicial sets (for details we refer to \cite{GJ}),
and introduce the concept of a relation of categories.

Let $[n]$ be the category with object set $\{0,1,...,n\}$ and a unique morphism $i\to j$ if $i\leq j$. Note that a functor $[m]\to[n]$ is the same as an order-preserving map from the set \(\{0,1,...,m\}\) to the set \(\{0,1,...,n\}\).
The object set of the \emph{simplex category} $\Delta$ consists of the
categories $[n]$, for $n\geq 0$. Morphisms in \(\Delta\) are functors between these
categories.

Consider the interval $[0,1]$ as a subspace of $\RR$. Given an integer \(n \ge 0\), the \emph{geometric \(n\)-simplex} is the subspace 
\(\Delta^n\) of \([0,1]^{n+1}\) consisting of tuples \(t = (t_0,\dots,t_n)\) with sum equal to \(1\). Denoting the standard basis for
\(\RR^{n+1}\) by \(e_0, \dots, e_n\), we may write \(t = t_0e_0 + \dots + t_n
e_n\).
Let $\Top$ denote the category of topological spaces. There is a functor \(\Delta \to \Top\),
\([n] \mapsto \Delta^n\) taking an order-preserving map \(f \colon [m] \to [n]\)
to the affine map \(f_* \colon \Delta^m \to \Delta^n\) with \(f_*(e_i) =
e_{f(i)}\)
for \(i = 0,1,\dots,m\).

A \emph{simplicial set} is a functor $Y:\Delta\op\to\set$, from the opposite
category of the simplex category to the category of sets, 
sending $[n]$ to the set $Y_n$ of $n$-simplices. Morphisms in the category of
simplicial sets, called \emph{simplicial maps}, are natural transformation of
functors \(\Delta\op\to\set\). We write $\sset$ for the category of simplicial
sets and simplicial maps.

The {\em geometric realization \(|Y|\)} of a simplicial set \(Y\) is the 
topological space given by the coequalizer diagram
\begin{displaymath}
    \coprod_{[m], [n]} Y_n \times \Delta([m],[n]) \times \Delta^m \rightrightarrows
    \coprod_{[n]} Y_n \times \Delta^n
    \rightarrow |Y|,
\end{displaymath}
where the two parallel horizontal maps take \((y,f,t) \in Y_n \times
\Delta([m],[n]) \times \Delta^m\) to \((Y(f)y, t)\) and \((y, f_*(t))\)
respectively.
A simplicial map $\phi:Y\to Y'$ is called a
\emph{weak equivalence} if the induced map $\vert\phi\vert:\vert Y\vert\to\vert
Y'\vert$ on geometric realization (see \cite[I.2]{GJ}) is a homotopy
equivalence.

A \emph{bisimplicial set} is a functor $X:\Delta\op\times\Delta\op\to\set$, sending the tuple $([m],[n])$ to the set $X_{m,n}$ of $(m,n)$-simplices.
As with simplicial sets, \emph{bisimplicial maps} $\phi:X\to X'$ are natural
transformations.
We write $\ssset$ for the category of bisimplicial sets and bisimplicial maps.


A \emph{relation (of categories)} from a small category \(\C\) to a small category \(\D\) is a functor of the form \(R \colon \R \to \C \times \D\).
A \emph{morphism of relations of categories} \(f \colon R \to R'\) from \(R
\colon \R \to \C \times \D\) to \(R' \colon \R' \to \C' \times \D'\) consists
of functors \(f_0 \colon \R \to \R'\), \(f_1 \colon \C \to \C'\) and
\(f_2 \colon \D \to \D'\) so that \((f_1 \times f_2) \circ R = R' \circ f_0\). 
We write $\rel$ for the category of relations of categories.

Given two categories $\C$ and $\D$ the \emph{twist isomorphism} 
\begin{equation*}
    \tw \colon \C \times \D \to \D \times \C
\end{equation*}
is the functor sending objects $(c,d)$ to $(d,c)$, and morphisms
$(\gamma,\delta)$ to $(\delta,\gamma)$. For a relation of categories \(R \colon
\R \to \C \times \D\), its \emph{transposed relation} \(R^T \colon \R \to \D
\times \C\) is the composition \(R^T = R \circ \tw\) of \(R\) and the twist isomorphism
\(\tw\). The \emph{transposition functor} $T:\rel\to\rel$ is the
functor \(R \mapsto R^T\).

\section{The Bisimplicial Rectangle Nerve}\label{sec:}
In this section we introduce the bisimplicial rectangle nerve of a relation. This is a bisimplicial set that is in a sense symmetric under transposition.

\begin{definition}
Let \(R \colon \R \to \C \times \D\) be a relation.
The \emph{bisimplicial rectangle nerve} $ER$ is the bisimplicial set whose $(m,n)$-simplices are 
functors \(r \colon [m] \times [n] \to \R\) such that there exist a necessarily unique
pair of functors $(a:[m]\to \C, b:[n]\to \D)$ with $a\times b = R\circ r$, that is, so that the following diagram commutes
\begin{equation*}
    \begin{tikzcd}
        {[m]\times[n]}\arrow[d,"r",swap]\arrow[dr,"a\times b"]&\\
        \R\arrow[r,"R"]&\C\times\D.
    \end{tikzcd}
\end{equation*}
If $\alpha:[m']\to[m]$ and $\beta:[n']\to[n]$ are order-preserving maps, then $ER(\alpha,\beta):ER_{m,n}\to ER_{m',n'}$ sends the $(m,n)$-simplex $r$ to the $(m',n')$-simplex \(r \circ (\alpha \times \beta)\). 
\end{definition}

A simplex \(r \colon [m] \times [n] \to \R\)
in the bisimplicial rectangle nerve can then be considered as a lift of a map of
rectangles \(a \times b
\colon [m] \times [n] \to \C \times \D\) to \(\R\).
This is the motivation for the name ``bisimplicial 
rectangle nerve''.

Let \(f \colon R \to R'\) be a morphism of relations given by
relations
\(R \colon \R \to \C \times \D\)
and 
\(R' \colon \R' \to \C' \times \D'\),
and functors 
\(f_0 \colon \R \to \R'\),
\(f_1 \colon \C \to \C'\) and
\(f_2 \colon \D \to \D'\).
There is a bisimplicial map 
\(Ef \colon ER \to ER'\) taking an \((m,n)\)-simplex \(r\) of \(ER\) to the \((m,n)\)-simplex
\(f_0 \circ r\) of \(ER'\).
It is straightforward to check that the assignment \(f \mapsto Ef\) gives us
a functor $E:\rel\to\ssset$.

Pre-composition with the twist isomoprhism 
\(\tw:\Delta\op\times\Delta\op\to\Delta\op\times\Delta\op\)
gives a functor $\tau:\ssset\to\ssset$.
Specifically, for $X\in\ssset$ the bisimplicial set
$\tau X$ is the composite functor
\begin{equation*}
    \begin{tikzcd}
        \Delta\op\times\Delta\op\arrow[r,"\tw"]&\Delta\op\times\Delta\op\arrow[r,"X"]&\set,
    \end{tikzcd}
\end{equation*}
and $\tau X_{m,n}=X_{n,m}$. 

If $R$ is a relation, then $r:[m]\times[n]\to\R$ is a simplex in $ER_{m,n}=\tau ER_{n,m}$ if and
only if the composition
\begin{equation*}
    \begin{tikzcd}
        {[n]\times[m]}\arrow[r,"\tw"]&{[m]\times[n]}\arrow[r,"r"]&\R
    \end{tikzcd}
\end{equation*}
is in $ER^T_{n,m}$. 
These maps give us a bijective bisimplicial map \(\tw^* \colon \tau ER \to ER^T\) taking 
\(r \in ER_{m,n}\) to \(r \circ \tw\).
It is natural in the sense that
given a morphism $f:R\to R'$, we have a commutative diagram
\begin{equation*}
    \begin{tikzcd}
        \tau ER\arrow[r,"\tw^*"]\arrow[d,"\tau Ef"]&ER^T\arrow[d,"Ef^T"]\\
        \tau ER'\arrow[r,"\tw^*"]&ER'^T. 
    \end{tikzcd}
\end{equation*}
We sum up this discussion in the following lemma:

\begin{lemma}\label{iso}
The map $\tw^* \colon \tau E\to ET$ is a natural isomorphism.\qed
\end{lemma}

Next, we consider the \emph{diagonal functor}
\begin{equation*}
    \oD:\Delta\op\to\Delta\op\times\Delta\op
\end{equation*}
where $\oD([n])=([n],[n])$ on objects and $\oD(\alpha)=(\alpha,\alpha)$ on morphisms \cite[p.197]{GJ}. Pre-composing with $\oD$ gives a functor $\od:\ssset\to\sset$ sending a bisimplicial set $X$ to its \emph{diagonal simplicial set} $\od(X) := X\circ \oD$ whose $n$-simplices $\od(X)_n$ are $X_{n,n}$.

Note that the diagram 
\begin{equation*}
    \begin{tikzcd}
        \Delta\op\arrow[r,"\oD"]\arrow[dr,"\oD",swap]&\Delta\op\times\Delta\op\arrow[d,"\tw"]\\
        &\Delta\op\times\Delta\op
    \end{tikzcd}
\end{equation*}
commutes. Since the diagonal functor $\od$ and the functor $\tau$ are defined by pre-composition of $\oD$ and $\tw$ respectively, we have:
\begin{lemma}\label{diagequal}
If $R$ is a relation of categories, then $\od(ER)=\od(\tau ER)$.\qed
\end{lemma}
Combining Lemma \ref{diagequal} with Lemma \ref{iso} we get:
\begin{lemma}\label{gamma}
    The map \(\tw^* \colon \tau E \to ET\) induces a natural isomorphism
    \(\od(\tw^*) \colon \od E \to \od ET\).
\end{lemma}

\section{The Functorial Dowker Duality Theorem}\label{bpdd}

In this section we finally introduce the Dowker nerve of a relation of
categories, and we state our version of the functorial Dowker duality theorem.
We have defined the bisimplicial rectangle nerve by
\begin{equation*}
    ER_{m,n}=\left\{r \colon [m] \times [n] \to \R \, \middle \vert \,
    \text{\(R \circ r\) is of the form \(a \times b \colon [m] \times [n] \to
\C \times \D\)} \right\}.
\end{equation*}
The \emph{nerve} of a small category $\C$ is the simplicial
set $\Ns\C$ whose $m$-simplices are functors from $[m]$ to $\C$, that is,
$\Ns\C_m=\cat([m],\C)$.

Note that given \(r \in ER_{m,n}\) with \(R \circ r = a \times b \colon [m] 
\times [n] \to \C \times \D\), the functors \(a \colon [m] \to \C\) and $b:[n]\to\D$ are 
uniquely determined by the universal property of products. In particular, there is a function \(\pi_R \colon ER_{m,n} \to N\C_m\)
given by \(\pi_R(r) = a\) for \(a \in N\C_m\) with \(R \circ r = a \times b \colon [m] \times [n] \to \C \times \D\).
\begin{definition}
    \label{thm:dowkerfiber}
    Given \(a \in N\C_m\), that is, a functor \(a \colon [m] \to \C\), the {\em
    fiber \(\pi_R^a\) of \(a\) under \(\pi_R\)} is the simplicial subset of
    the simplicial set \([n] \mapsto ER_{m, n}\) consisting of functors \(r
    \colon [m] \times [n] \to \R\) such that there exist a functor \(b \colon
    [n] \to \D\) with \(R \circ r = a \times b\).
\end{definition}

In order to state our version of Dowker duality we introduce the concept of a Dowker relation.
\begin{definition}
    \label{thm:dowkerrelation}
    A {\em Dowker relation} is a relation \(R \colon \R \to \C \times \D\)
    with the property that for every \(a \in N\C_m\), the fiber \(\pi_R^a\) of
    \(a\) under \(\pi_R\) is contractible or empty.
\end{definition}
In Section \ref{sec:appl} we look at concrete Dowker relations, one class of which is described in Corollary \ref{cor: initial terminal}.

\begin{definition}
    The {\em Dowker nerve} of the relation $R \colon \R\to \C\times \D$  is
    the simplicial set $DR$ whose set of $m$-simplices
    \(DR_m\) is the image of the map \(\pi_R \colon ER_{m,0} \to N\C_m\).
\end{definition}

Let \(f \colon R \to R'\) be a morphism of relations
given by
relations
\(R \colon \R \to \C \times \D\)
and 
\(R' \colon \R' \to \C' \times \D'\),
and functors 
\(f_0 \colon \R \to \R'\),
\(f_1 \colon \C \to \C'\) and
\(f_2 \colon \D \to \D'\). 
The assignment \(a \mapsto Df(a) = f_1\circ a\)
defines a simplicial map \(Df \colon DR \to DR'\), so we have a functor
$D:\rel\to\sset$.
\begin{remark}
    The $m$-simplices \(DR_m\) contain the image of \(\pi_R \colon ER_{m,n}
    \to N\C_m\) for all \(n\geq 0\).  We also write \(\pi_R \colon ER_{m,n} \to
    DR_m\) for the map \(\pi_R\) with \(DR_m\) as codomain instead of
    \(N\C_m\). 
    Fixing \(m\) we obtain the simplicial sets $X$ and $Y$ where \(X_n:=ER_{m,n}\) and \(Y_n:=DR_m\) as a constant simplicial set.
    Since the connected components\footnote{The connected components of the simplicial set $X$ are the graph-components of the multigraph $X_1\rightrightarrows X_0$.} of a constant simplicial set are given by
    degenercies of zero simplices, the relation \(R\) is a Dowker relation if
    and only if the simplicial map $\pi:X\to Y$, which on $n$-simplices is 
    \(\pi_R \colon ER_{m,n} \to DR_m\),
    is a weak equivalence.
\end{remark}

\begin{theorem}[Dowker Equivalence]\label{Dowker Duality Intro}
    \label{thm:dowkerduality}
    If \(R\) is a Dowker relation, then the projection maps
    \[\pi_{R} \colon ER_{m,n} \to DR_m\]
    induce a weak equivalence \(\od(\pi_{R}) \colon
    \od(ER) \to DR\) of simplicial sets.
\end{theorem}
\begin{proof}
    Fixing $m\geq 0$, let \(A\) be the simplicial set with \(A_n = ER_{m,n}\), and consider \(DR_m\) as a constant simplicial set. By our assumption on \(R\) the simplicial map $A\to DR_m$, which on $n$-simplices is \(\pi_R \colon ER_{m,n} \to DR_m\),
    is a weak equivalence.
    Let $B$ be the bisimplicial set $B_{m,n}:= DR_m$ constant in one direction and consider \(\pi_R \colon ER_{m,n} \to DR_m\) as a bisimplicial map \(\pi_R \colon ER \to B\). 
    By \cite[Prop.
    IV.1.7]{GJ}, attributed to Tornehave in \cite{quillen}, the projection map $\pi_R$ induces a weak equivalence of
    diagonals. That is, the map $\od(\pi_R):\od(ER)\to d(B)=DR$ is a weak equivalence.
\end{proof}

We now use Theorem \ref{thm:dowkerduality} to prove a functorial Dowker duality theorem for relations of categories. Consider the \emph{left projection functor}
\begin{equation*}
    \oP:\Delta\op\times\Delta\op\to\Delta\op
\end{equation*}
sending objects $([m],[n])$ to $[m]$ and morphisms $(\alpha,\beta)$ to $\alpha$. We consider the functor $\p:\sset\to\ssset$, sending a simplicial set $Y:\Delta\op\to\set$ to the composite functor
\begin{equation*}
    \begin{tikzcd}
        \Delta\op\times\Delta\op\arrow[r,"\oP"]&\Delta\op\arrow[r,"Y"]&\set.
    \end{tikzcd}
\end{equation*}
Note that $\oP\circ\oD$ is the identity, so we have $\od(\p\,(Y))=Y$.

There is a natural transformation $\pi:E\to \p  D$, 
so that $\pi_R:ER\to \p  DR$ 
takes \(r \colon [m]\times [n] \to \R\) with \(R \circ r = a \times b\) to \(a\).
This means that for each morphism of relations $f:R\to R'$ we have a commutative square
\begin{equation}\label{nat.trans}
    \begin{tikzcd}
    ER\arrow[d,"Ef",swap]\arrow[r,"\pi_R"]&\p  DR\arrow[d,"\p  Df"]\\
    ER'\arrow[r, "\pi_{R'}"]&\p  DR'.
    \end{tikzcd}
\end{equation}
For the transposed morphism $f^T:R^T\to R'^T$, we get a commutative square 
\begin{equation}\label{nat.trans2}
    \begin{tikzcd}
    ER^T\arrow[d,"Ef^T",swap]\arrow[r,"\pi_{R^T}"]&\p  DR^T\arrow[d,"\p  Df^T"]\\
    ER'^T\arrow[r, "\pi_{R'^T}"]&\p  DR'^T.
    \end{tikzcd}
\end{equation}
In this way, we may regard $\pi$ as a natural transformation $\pi \colon ET\to \p  DT$ as well.
\begin{theorem}[Dowker Duality]
    \label{thm:dowkerequivalence}
Given a morphism $f:R\to R'$ of relations of categories, there is a commutative diagram of the form
\begin{equation*}
    \begin{tikzcd}
    DR\arrow[d,"Df",swap]&\arrow[l,"\od(\pi_{R})",swap]\od(ER)\arrow[d,"\od(Ef)"]\arrow[r,"\od(\tw^*)"]&\od(ER^T)\arrow[d,"\od(Ef^T)"]\arrow[r,"\od(\pi_{R^T})"]&DR^T\arrow[d,"Df^T"]\\
    DR'&\arrow[l,"\od(\pi_{R'})",swap]\od(ER')\arrow[r,"\od(\tw^*)"]&\od(ER'^T)\arrow[r,"\od(\pi_{R'^T})"]&DR'^T.
    \end{tikzcd}
\end{equation*}
    If the relations \(R, R^T, R'\) and \({R'}^T\) are Dowker relations, then all
    horizontal maps in the diagram are weak equivalences of simplicial sets.
\end{theorem}
\begin{proof}
Applying the diagonal $\od$ to the commutative squares \eqref{nat.trans} and \eqref{nat.trans2}, together with Lemma \ref{gamma}, we see the diagram commutes.
By Lemma \ref{gamma}, the maps labeled \(\od(\tw^*)\) are isomorphisms.
The statement about weak equivalences is a now consequence of Theorem
\ref{thm:dowkerduality}.
\end{proof}

\section{Dowker Nerves and Homotopies}\label{sec:homotopies}
We look at mophisms of relations of categories that induce homotopies when taking the Dowker nerve.

To talk about homotopies, we need to define the product of relations. Let $R:\R\to\C\times\D$ and $R':\R'\to\C'\times\D'$ be relations of categories. We have projections $\pi_\C:\C\times\D\to\C$ and $\pi_\D:\C\times\D\to\D$. The product $R\times R'$ in $\rel$ is the relation $\R\times\R'\to(\C\times\C')\times(\D\times\D')$ sending $(x,x')$ to $((\pi_\C R(x),\pi_{\C'} R' (x')),(\pi_{\D} R(x),\pi_{\D'} R' (x')))$. Projections to first and second factors give the two structure maps for the product.

\begin{lemma}\label{product preserved}
    The functor $D:\rel\to\sset$ preserves products, so given two relations $R:\R\to\C\times\D$ and $R':\R'\to\C'\times\D'$ of categories, the projections 
    onto \(R\) and \(R'\) induce an isomorphism
    \(D(R \times R') \xrightarrow{\cong} D(R) \times D(R')\).
\end{lemma}
\begin{proof}
    A simplex $r\in DR_m$ is a map $r:[m]\times[0]\to\R$ with the property that there exist $a:[m]\to\C$ and $b:[0]\to\D$ so that $R\circ r=a\times b$. It is uniquely defined by a map $\Tilde{r}:[m]\to\R$ such that the composition
    \begin{equation*}
        [m]\xrightarrow{\Tilde{r}}\R\xrightarrow{R}\C\times\D\xrightarrow{\pi_\D}\D
    \end{equation*}
    is constant. Explicitly, given \(r, a\) and $b$, we define $\Tilde r$ by $\Tilde{r}(i)=r(i,0)$. Conversely, given $\Tilde r$, the maps \(r, a\) and \(b\) are given by letting $r(i,0) =\Tilde{r}(i)$, $a=\pi_\C\circ R\circ\Tilde{r}$ and $b(0) = \pi_\D\circ R\circ\Tilde{r}(i)$.

    A simplex in $D(R\times R')_m$ is uniquely defined by a map $(\Tilde{r},\Tilde{r}'):[m]\to \R\times \R'$ where $\pi_{\D\times\D'}\circ(R\times R')\circ (\Tilde{r},\Tilde{r}') $ is constant, which is equivalent to $\pi_\D \circ R\circ\Tilde{r}$ and $\pi_{\D'}\circ R' \circ \Tilde{r}'$ both being constant. Thus, under the isomorphism \(\cat([m], \R \times \R') \xrightarrow{\cong} \cat([m], \R) \times \cat([m], \R')\), \(m\)-simplices of \(D(R \times R')\) are taken bijectively to \(m\)-simplices of \(D(R) \times D(R')\).   
\end{proof}
The following is a consequence of the fact that the Dowker nerve of a relation of the form \(\id_{\C \times \D} \colon \C \times \D \to \C \times \D\) is equal to the nerve of the category \(\C\).
\begin{lemma}\label{nerveof1 lemma}
    Given \(n \ge 0\),
    the Dowker nerve of the relation \({\id_{[n]\times[0]}} \colon [n] \times [0] \to [n] \times [0]\)
    is the simplicial \(n\)-simplex \(\Delta{[n]}\).\qed
\end{lemma}

For $i= 0,1$ we have the morphism of relations $d^i:\id_{[0]\times[0]}\to\id_{[1]\times[0]}$ where the map $d^i_0:[0]\times[0]\to [1]\times[0]$ does not hit $(i,0)$, the map $d^i_1:[0]\to[1]$ does not hit $i$ and $d^i_2=\id_{[0]}$. 

\begin{definition}
    Given two morphisms of relations $f^0,f^1: R\to R'$, a \emph{\transformation} $H$ from $f^1$ to $f^0$ is a morphism $H:\id_{[1]\times[0]}\times R \to R'$ of relations such that the diagram
    \begin{equation}\label{transformationdiagram}
        \begin{tikzcd}
            \id_{[0]\times[0]}\times R\arrow[rr,"d^i\times \id_R"] && \id_{[1]\times[0]}\times R\arrow[d,"H"]\\
            R\arrow[rr,"f^i"]\arrow[u,"\cong"] && R'
        \end{tikzcd}
    \end{equation}
    commutes for $i=0,1$.
\end{definition}

\begin{proposition}\label{homotopy prop}
    Given two morphisms of relations $f^0,f^1: R\to R'$ and a {\transformation} $H$ from $f^1$ to $f^0$, the maps $Df^0,Df^1:DR\to DR'$ are homotopic.
\end{proposition}
\begin{proof}
    Taking the Dowker nerve of diagram \eqref{transformationdiagram} using Lemma \ref{product preserved} and Lemma \ref{nerveof1 lemma}, for $i=0,1$, we get a commutative diagram of the form
    \begin{equation*}
        \begin{tikzcd}
            \Delta{[0]}\times DR\arrow[rr,"d_i\times \id_{DR}"] && \Delta{[1]} \times DR\arrow[d,"\widehat{DH}"]\\
            DR\arrow[rr,"Df^i"]\arrow[u,"\cong"] && DR'.
        \end{tikzcd}
    \end{equation*}
    The map $\widehat{DH}$ is the desired homotopy.
\end{proof}

\section{Applications}\label{sec:appl}
In order to justify our categorification of the Dowker Theorem, we 
show how it is related to Quillen's Theorem A and different versions of
the singular simplicial set of a simplicial complex.
\subsection{Quillen's Theorem A}
We first apply Theorem \ref{thm:dowkerduality} to prove Quillen's Theorem A. Given functors \(F \colon \C \to \A\) and \(G \colon \D \to \A\), the \emph{comma
category} \(F\downarrow G\) has objects given by triples \((c,d,f)\), where $c\in\C$, $d\in\D$ and \(f\) is a morphism \(f \colon Fc \to
Gd\) in $\A$. A morphism in \(F \downarrow G\), of the form \((c,d,f) \to
(c',d',f')\),
consists of morphisms \(\alpha_L \colon c \to c'\) and \(\alpha_R \colon d \to
d'\) such that the following diagram commutes:
\begin{equation*}
    \begin{tikzcd}
        Fc \arrow[d,"F\alpha_L"] 
        \arrow[r, "f"] & \arrow[d,"G\alpha_R"] Gd \\
        Fc' \arrow[r, "f'"] & Gd'.
    \end{tikzcd}
\end{equation*}

The projection \(R \colon F \downarrow G \to \C \times \D\) sending $(c,d,f)$ to $(c,d)$ is a relation. 
Let \(r \colon [m] \times [n] \to F \downarrow G\) be a \((m,n)\)-simplex of $ER$ with \(a \colon [m] \to
\C\) and \(b \colon [n] \to \D\) satisfying the equation \(R \circ r = a \times
b\).
Given \((i,j) \in [m] \times [n]\), we write \(r(i,j) = (a(i), b(j), f_{ij} \colon
Fa(i) \to Gb(j))\).
Note that \(f_{ij} = f_{mj} \circ Fa(i \to m)\), and that \(f_{mj} = Gb(0 \to j) \circ f_{m0}\), so
the morphism \(f_{ij} \colon Fai \to Gbj\)
is equal to the composition
\begin{displaymath}
    Fa(i) \xrightarrow{Fa(i \to m)} Fa(m) \xrightarrow{f_{m0}} Gb(0) \xrightarrow{Gb(0 \to j)} Gb(j).
\end{displaymath}

This means that \(r\) is uniquely determined by \(a\), \(b\) and \(f_{m0}
\colon Fa(m) \to Gb(0)\).
With this in mind we see that the fiber of \(a \in DR_m\) under \(\pi_R:ER_{m,n} \to DR_m\) is
isomorphic to the nerve \(N(Fa(m) \downarrow G)\) of the comma category \(Fa(m)
\downarrow G\) for the functors \(Fa(m) \colon * \to \D\) and \(G\). Similarly,
the fiber of \(b \in DR^T_n\) under \(ER_{m,n}
\to DR^T_n\) is
isomorphic to nerve \(N(F\downarrow Gb(0))\) of the comma category \(F \downarrow Gb(0)\)
for the functors \(F\) and \(Gb(0) \colon * \to \C\).

Specializing to
$\A=\D$ and $G=\id_\D$ we can prove the following:
\begin{corollary}[Quillen's Theorem A \cite{quillen}] Consider a functor
    $F:\C\to\D$. If $N(F\downarrow d)$ is contractible for every object $d\in \D$, then the
    map $NF:N\C\to\ N\D$ is a weak equivalence.
\end{corollary}
\begin{proof}
Consider the comma category $F\downarrow\id_\D$. Note that the category
$Fa(m)\downarrow\id_\D$ has initial object $(*,Fa(m),\id_{Fa(m)})$, so the nerve
$N(Fa(m)\downarrow\id_\D)$ is contractible, implying that the 
projection $R:F\downarrow\id_\D\to\C\times\D$ is a Dowker relation.
Furthermore, by the preceding discussion, if all $N(F\downarrow
d)$ are contractible, then also the transpose \(R^T\)
is a Dowker relation.
Since \(r \in ER_{m,0}\) is uniquely determined by \(a \colon [m]\to \C\),
\(d \in \ob \D\) and \(f \colon Fa(m) \to d\),
the set of $m$-simplices of
the nerve
$DR$ consists of functors $a:[m]\to\C$ such that there exists an object $d$ in
$\D$ and a morphism $f:Fa(m)\to d$.
We can always choose $d=Fa(m)$ and
$f=\id_{Fa(m)}$, so
$DR=N\C$.
The simplices in $DR^T_n$ similarly are functors $b:[n]\to\D$ such that there
is an object $c$ in $\C$ and a functor $f:Fc\to b(0)$, but such a triple $(c,*,f)$
is an object in $F\downarrow b(0)$ which is non-empty by the assumption that
$N(F\downarrow b(0))$ is contractible. So we get that $DR^T=N\D$ and  $N\C\simeq N\D$.
We still need to show that $NF$ is a weak equivalence.

Consider the projections $\pi_\C:F\downarrow \id_\D\to\C$ and $\pi_\D:F\downarrow \id_\D\to\D$ sending $(c,d,f)$ to $c$ and $d$ respectively. We have the commuting diagram
\begin{equation}\label{quillenfactorization}
    \begin{tikzcd}
        &{\od(ER)} \arrow[d,"{\oD^*}"] \arrow[dl,"\od(\pi_{R})",swap]
        \arrow[r, "\od(\tw^*)"]&
        \od(ER^T) \arrow[d, "\od(\pi_{R^T})"]\\
        N\C& N(F\downarrow\id_\D)\arrow[l,"N\pi_\C"]\arrow[r,"N\pi_\D",swap]&N\D,
    \end{tikzcd}
\end{equation}
where $\oD^* \colon \od(ER)_m = ER_{m,m} \to N(F\downarrow\id_\D)_m$ 
is precomposition with the diagonal functor
$\oD:[m]\to[m]\times[m]$. Furthermore, there is a natural map
$\eta:F\circ\pi_\C\to\pi_\D$ with components $\eta_{(c,d,f)}=f:Fc\to d$
for objects $(c,d,f)$ in $F\downarrow \id_\D$. This induces a homotopy on
nerves $NF\circ N\pi_\C\simeq N\pi_\D$. Using diagram
\eqref{quillenfactorization} we get  $NF\circ\od(\pi_{R})\simeq\od(\pi_{R^T}) \circ \od(\tw^*)$.
The map $\od(\tw^*)$ is an isomorphism. Since $R$ is a Dowker relation, the maps $\od(\pi_{R})$ and
$\od(\pi_{R^T})$ are weak equivalences, and therefore so is $NF$.
\end{proof}
The above argument is very close to the proof in \cite{quillen}. Arguably, 
the proof in \cite{quillen} is more elegant than the proof presented here. The point we are making is that Quillen's Theorem A and the Dowker duality of Theorem \ref{thm:dowkerequivalence} are closely connected.

\subsection{Simplicial Sets from Simplicial Complexes}\label{sec:example}
In this subsection we look at two ways of turning simplicial complexes into simplicial sets resembling the singular complex of a topological space. We use the Dowker duality of Theorem \ref{thm:dowkerequivalence} to prove that one of these singular complex constructions is functorial and that it is of the correct homotopy type. We begin by investigating relations given by inclusions of full subcategories.

\begin{definition}
    Let $R:\R\to\C\times\D$ be a relation with \(R\) an inclusion of a full subcategory. Given $a\in N\C_m$, we let $\D^a_R\subseteq\D$ be the full subcategory consisting of all objects $d \in \D$ such that \((a(i), d) \in \R\) for \(i = 0, 1, \dots, m\).
\end{definition}
\begin{lemma}\label{nerveremark}
    Let $R:\R\to\C\times\D$ be a relation with \(R\) an inclusion of a full subcategory. Given $a\in N\C_m$, the projection \(\pi_R^a \to N\D\) taking \(r \colon [m] \times [n] \to \R\) to the uniquely determined \(b \colon [n] \to \D\) such that \(R \circ r = a \times b\) induces a bijection \(\pi_{R}^a \to N\D^a_R\) of simplicial sets.
\end{lemma}
\begin{proof}
    By construction, the given \(a\) and \(r\) as in the statement, the uniquely determined \(b \colon [n] \to \D\) takes values in \(\D^a_R\). Thus we have an induced function \(\pi_{R}^a \to N\D^a_R\) of simplicial sets. Since \(R\) is an inclusion, the assignment \(r \mapsto b\) is injective. For surjectivity, note that by construction, given \(b \in N\D^a_R\), the functor \(a \times b \colon [m] \times [n] \to \C \times \D\) factors through \(\R\).
\end{proof}

    


The nerve of a category with either initial or terminal object is contractible \cite[p.8]{quillen}, so we have the following corollary.

\begin{corollary}\label{cor: initial terminal}
    Let $R:\R\to\C\times\D$ be the inclusion of a full subcategory. If $\D$ has the property that all full subcategories have an initial or terminal object, then $R$ is a Dowker relation. \qed
\end{corollary}
We see two examples of such categories below, namely categories that are totally ordered sets and the translation category of a set.

A \emph{simplicial complex} $(K,V)$ is a set $V$ and a set $K$ of finite 
subsets of
$V$ such that $\sigma\in K$ and $\tau\subseteq\sigma$ implies $\tau\in K$.
We follow standard terminology and say that \(K\) is a simplicial complex, leaving the vertex set \(V\) implicit. Note that inclusion $\subseteq$ is a partial order on $K$ making it a partially ordered set $K_\subseteq$. 

Consider the topological space $[0,1]^S$ whose elements are functions from a set $S$ to the interval $[0,1]\subseteq\RR$. If \(S\) is finite, then \([0,1]^{S} = \prod_S [0,1]\) is given the product topology.
If \(S\) is infinite, then \([0,1]^{S}\) is given the topology where \(U \subseteq [0,1]^{S}\) is open if and only if for every finite subset \(W\) of \(S\) the set \(U \cap [0,1]^{W}\) is open in \([0,1]^{W}\). The \emph{support} of a function $S\to[0,1]$ is the subset of $S$ consisting of the elements that give non-zero values of the given function.

The {\em geometric realization \(|K|\)} of a simplicial complex \((K,V)\) is the subspace of \([0,1]^{V}\) consisting of functions \(\alpha \colon V \to [0,1]\) satisfying firstly that its support is a simplex in $K$ and secondly that the sum of its values is equal to \(1\), that is \(\sum_{v \in V} \alpha(v) = 1\).

We consider two ways of constructing a simplicial set from a simplicial complex $(K,V)$. The \emph{singular complex} on \(K\) is the simplicial set $\sing(K)$ whose set of $m$-simplices are 
\begin{equation*}
    \sing(K)_{m}=\{a:\{0,1,\dots,m\}\to V\, \vert \,
    \{a(0),a(1),\dots,a(m)\}\in K\}.
\end{equation*}
The simplicial structure on $\sing(K)$ is induced from the cosimplicial set \([m] \mapsto \{0,\dots, m\}\) given by forgetting the order on \([m]\).

Suppose that the simplicial complex \((K,V)\) has a total order $\leq$ on \(V\). The \emph{ordered singular complex} on \(K\) is the simplicial set $\sing_\leq(K)$ with set of $m$-simplices given by order-preserving maps, that is,
\begin{equation*}
        \sing_\leq (K)_{m}=\{a:[m]\to V_\leq\, \vert \,
    a([m])\in K\}.
\end{equation*}
This simplicial set is a simplicial subset of the nerve of the category \(V_{\le}\).

\begin{remark}
    The functor \(K \mapsto \sing (K)\) from simplicial complexes to simplicial sets is right adjoint to a functor \(X \mapsto MX\). Here \(MX\) is the simplicial complex with vertex set \(X_0\) and 
    with simplices given by vertex sets of simplices of the simplicial set \(X\). Note that by the vertex set of \(x \in X_n\) we mean the set of zero-dimensional faces of \(x\). Moreover, the functor \(K \mapsto \sing_{\le} (K)\) from ordered simplicial complexes to the category of simplicial sets with a total order on the set of \(0\)-simplices also has a right adjoint functor.
\end{remark}

We now explain how $\sing_\le(K)$ and $\sing(K)$ can be considered as Dowker nerves of relations.
\begin{enumerate}
    \item Assume that $V$ has a total order $\leq$ making it a totally ordered set $V_\le $. Consider the full subcategory $\R_1\subseteq V_\le\times K_\subseteq$ where $(v,\sigma)\in\ob{\R_1}$ if and only if $v\in\sigma$. The inclusion $R_1:\R_1\to V_\le \times K_\subseteq$ is a relation, and $D R_1 = \sing_\le (K)$.
    \item 
    The \emph{translation category} $\Vtr$ of \(V\) has object set $\ob{\Vtr}=V$ and a unique morphism $v\to w$ between any pair of objects $v,w\in V$. Consider the full subcategory $\R_2\subseteq\Vtr\times K_\subseteq$ where $(v,\sigma)\in\ob{\R_2}$ if and only if $v\in\sigma$. The inclusion $R_2:\R_2\to\Vtr\times K_\subseteq$ is a relation, and $DR_2=\sing(K)$.
\end{enumerate}

Note that for any choice of order $\leq$ on the vertex set $V$ of a simplicial complex $(K,V)$ we have an injective map $\sing_\le(K)\hookrightarrow\sing(K)$ induced by the inclusion $V_\leq\hookrightarrow\Vtr$. 

We define a map \(\varphi \colon |\sing (K)| \to |K|\). Every element in \(|\sing (K)|\) is represented by a pair \((a,t) \in \sing(K)_m
\times \Delta^m\).
Given such a pair \((a,t)\) with \(t = (t_0, \dots, t_m)\), let \(a_*(t) \colon
V \to [0,1]\) be the
element of \(|K|\) with \(a_*(t)(v) = 
\sum_{a(i) = v} t_i\).
It is straight-forward to verify that $(a, t) \mapsto \varphi(a, t) = a_*(t)$ defines a natural continuous map \(\varphi \colon
|\sing(K)| \to |K|\).
Given a total order \(\le\) on the vertex set \(V\) of \(K\), we denote by $\varphi_\leq:|\sing_\leq(K)|\to|K|$ the map given by the composition
\begin{equation*}
    |\sing_\leq(K)| \hookrightarrow|\sing(K)|\xrightarrow{\varphi} |K|.
\end{equation*}
The following is well-known (stated by Milnor in \cite[p.358]{milnor} and Curtis in \cite[p.118]{curtis}).
\begin{proposition}\label{singrealiz}
    Let \(K\) be a simplicial complex with a total order \(\le\) on the vertex set \(V\).
    The map \(\varphi_\leq \colon |\sing_\leq (K)| \to |K| \) is a homeomorphism.
\end{proposition}
\begin{proof}
    We first consider the
    situation where \(V\) is finite. If \(V\) has cardinality \(m + 1\),
    then \(V_{\le}\) is isomorphic to the ordinal \([m]\) by
    an order-preserving bijection \(\gamma \colon [m] \to V\).
    Given an element \(\alpha \colon V \to [0,1]\) of \(|K|\), the composition
    \(\alpha \circ \gamma \colon [m] \to [0,1]\) is an element of \(\Delta^m\),
    so the pair \((\gamma, \gamma \circ \alpha) \in \sing_{\le}(K)_m \times
    \Delta^m\)
    represents an element \(\psi(\alpha) \in |\sing_\le(K)|\).
    This defines a continuous map \(\psi \colon |K| \to |\sing_\le(K)|\).
    A direct verification yields that \(\varphi_\leq\) and \(\psi\) are inverse of
    each other, and thus they are homeomorphisms.
    
    If \(V\) is not finite, given a finite subset \(W\)
    of \(V\) we let \(W_\le\) be the total order induced from \(V_\le\), and we
    let \(K_W\) be the simplicial complex on the vertex set \(W\) consisting 
    of subsets of \(W\) contained in \(K\). Then \(K\) is the union of the 
    simplicial complexes \(K_W\) for \(W\) a finite subset of \(V\) and 
    \(\sing_\le(K) = \bigcup_{W \subseteq V} \sing_\le(K_W)\), where the union
    is taken over all finite subsets of \(V\). That \(\varphi_\leq \colon |\sing_\le(K)| \to |K|\) is a homeomorphism now follows from the fact that both kinds of geometric 
    realization are given a topology that commutes with unions, and that \(\varphi_\leq = \bigcup_{W \subseteq V} \varphi_\leq^{W}\), where the union is taken over all finite subsets of \(V\) 
    and \(\varphi^W_\leq \colon |\sing_\le(K_W)| \to
    |K_W|\) is the restriction of \(\varphi_\leq\) to \(|\sing_\le(K_W)|\).
\end{proof}

The geometric realization of the simplicial set $\sing_\leq (K)$ is homeomorphic to the geometric realization of the simplicial complex $(K,V)$ but choosing an order on \(V\) breaks functoriality. The simplicial set $\sing(K)$ is functorial in \((K,V)\), but its geometric realization is not homeomorphic to the geometric realization of $(K,V)$. However, we proceed to show that they are homotopy equivalent.

Consider the full subcategory \(\R_0 \subseteq (\V \times V_\le) \times K_\subseteq\) consisting of pairs \(((v,w), \sigma)\), where both \(v\) and \(w\) are vertices of the simplex \(\sigma\). Let \(R_0 \colon \R_0 \to (\V \times V_\le ) \times K_\subseteq\) be the inclusion relation. The projections \(\V \times V_\le \to V_\le\) and \(\V \times V_\le \to \V\) induce morphisms of relations \(R_0 \to R_1\) and \(R_0 \to R_2\) giving, by Theorem \ref{thm:dowkerequivalence}, a commutative diagram of the form

\begin{equation}
    \label{complexdiagram}
    \begin{tikzcd}
        DR_1  & \od(ER_1) \arrow[l, "\od(\pi_{R_1})", swap]
        \arrow[r, "\od(\tw^*)"] &
        \od(ER_1^T)  \arrow[r, "\od(\pi_{R_1^T})"] & DR_1^T \\
        DR_0 \arrow[u] \arrow[d] & \od(ER_0) \arrow[l, "\od(\pi_{R_0})", swap]
        \arrow[d] \arrow[u] \arrow[r, "\od(\tw^*)"] &
        \od(ER_0^T) \arrow[d] \arrow[u] \arrow[r, "\od(\pi_{R_0^T})"]  &
        DR_0^T \arrow[d] \arrow[u] \\
        DR_2  & \od(ER_2)  \arrow[l, "\od(\pi_{R_2})", swap]
        \arrow[r, "\od(\tw^*)"] &
        \od(ER_2^T)  \arrow[r, "\od(\pi_{R_2^T})"] & DR_2^T.
    \end{tikzcd}
\end{equation}

We show that all relations appearing in diagram \eqref{complexdiagram} are Dowker relations so that all horizontal maps are weak equivalences. 
The categories $V_\le$, $\Vtr$ and \(\V \times V_{\le} \) have the property that every full subcategory has an initial object, so by Corollary \ref{cor: initial terminal} we conclude that that $R_1^T$, $R_2^T$ and $R_0^T$ are Dowker relations. 
Next, let $a:[m]\to V_\le $ be a functor whose image is a simplex in $K$. By Lemma \ref{nerveremark} the fiber $\pi^a_{R_1}$ is isomorphic to the nerve of the category $(K_\subseteq)_{R_1}^a$ consisting of all simplices that contain the image of $a$. The simplex $a([m])$ is an initial object in $(K_\subseteq)_{R_1}^a$, so the fiber is contractible and $R_1$ is a Dowker Relation. 
Similarly, the nerve $\Ns(K_\subseteq)_{R_2}^{a}$ is contractible for every $a\in (DR_2)_m$ and the nerve $\Ns(K_\subseteq)_{R_0}^{a}$ is contractible for every $a\in (DR_0)_m$, making $R_2$ and $R_0$ Dowker relations as well. 

The rightmost vertical maps in diagram \eqref{complexdiagram} are identity maps, thus we can conclude that the maps \(\sing_\leq(K)=DR_1 \leftarrow DR_0 \to DR_2=\sing(K)\) are weak equivalences. Finally, consider the (non-commutative) diagram
\begin{equation} \label{smalltriangle}
\begin{tikzcd}
    R_0\arrow[r]\arrow[d] & R_2 \\
    R_1\arrow[ur,hook] &
\end{tikzcd}
\end{equation}
The functor $([1]\times[0])\times\R_0\to \R_2$ defined on objects by
\begin{equation*}
    ((i,0),((v,w),\sigma))\mapsto 
    \begin{cases}
    (v,\sigma) & \text{if }i=0\\
    (w,\sigma) & \text{if }i=1
    \end{cases}
\end{equation*}
induces a {\transformation} $H:\id_{[1]\times[0]}\times R_0\to R_2$ from the top path in diagram \eqref{smalltriangle} to the bottom path, and so by Proposition \ref{homotopy prop} the two paths after taking the Dowker nerve are homotopic. In particular, the top-left triangle in the diagram
\begin{equation*}\label{naturaldiagram}
    \begin{tikzcd}
        {\vert DR_0\vert} \arrow[r,"\simeq"]
        \arrow[d,"\simeq",swap] & 
        {\vert\sing(K)\vert}
        \arrow[d,"\varphi"] \\ 
        \vert\sing_\leq(K)\vert
        \arrow[ur,hook]\arrow[r,"\varphi_\leq"] \arrow[r,"\cong",swap] &
        \vert K\vert.
    \end{tikzcd}
\end{equation*}
commutes up to homotopy. By construction the triangle at the bottom-right commutes, so we can conclude the following:

\begin{corollary}\label{singrealiz2}
    Let \(K\) be a simplicial complex. The map \(\varphi \colon |\sing (K)| \to |K| \) is a homotopy equivalence and it is natural in \(K\). \qed
\end{corollary}

In \cite{omar} it is proven that the inclusion \(|\sing_{\le} (K)| \to |\sing (K)|\) is a homotopy equivalence. 
However the lack of functoriality of \(\sing_{\le} (K)\) prevents \cite{omar} from stating the result about naturality in Corollary \ref{singrealiz2}.

This result can be related to topological data analysis since
given a filtered simplicial complex $\{K_\alpha\}_{\alpha\in A}$ we obtain a filtered simplicial set $\{\sing(K_\alpha)\}_{\alpha\in A}$.
The filtered topological spaces obtained by taking geometric realizations of these two filtrations are of the same homotopy type. In particular, they have isomorphic
persistent homology.

\bibliographystyle{plainurl}
\bibliography{biblio}
\end{document}